\documentclass[oneside]{amsart}
\usepackage{amsmath,amssymb,amsfonts,amsthm}
\usepackage{color}
\usepackage{graphicx}
\usepackage{geometry} 
\usepackage{amscd}
\usepackage{enumitem}

\title{$(\alpha,\beta)$-metrics satisfying the $T$-condition or the $\sigma T$-condition}

\author{Salah G. Elgendi and  L\'aszl\'o Kozma}

\address{S. G. Elgendi, Department of Mathematics, Faculty of Science, Benha
 University, Egypt} \email{salah.ali@fsci.bu.edu.eg, \, salahelgendi@yahoo.com}
 \urladdr{http://www.bu.edu.eg/staff/salahali7}

\address{L\'aszl\'o Kozma, Department of Geometry, Institute of Mathematics, University of Debrecen,
 H-4002 Debrecen, P. O. Box 400,  Hungary} \email{kozma@unideb.hu}
 \urladdr{http://www.math.unideb.hu/kozma-laszlo/}



\def\blue#1{\textcolor[rgb]{0.0,0.0,1.0}{#1}}

\newcommand{\T}{{\mathcal T}}

\newcommand{\Real}{\mathbb R}

\def\+{\!+\!}

\def\={\!=\!}
\def\<{\!<\!}
\def\>{\!>\!}

\newcommand\undersym[2]{\raisebox{-7pt}{\tiny$#2$}{\kern-8pt}\mbox{$#1$}}
\newcommand\undersymm[2]{\raisebox{-7pt}{\tiny$#2$}{\kern-15pt}\mbox{$#1$}}

 \let\oldmarginpar\marginpar
\renewcommand\marginpar[1]{\oldmarginpar[\raggedleft\footnotesize #1]%
  {\blue{\raggedright \footnotesize \fbox{
      \begin{minipage}{1.0\linewidth}
        #1
      \end{minipage}
}}}}

\numberwithin{equation}{section} 
\numberwithin{figure}{section} 

\theoremstyle{plain}
\newtheorem*{theorem*}{Theorem}
\newtheorem{theorem}{Theorem}[section]
\newtheorem{lemma}[theorem]{Lemma}
\newtheorem{proposition}[theorem]{Proposition}
\newtheorem{corollary}[theorem]{Corollary}
\theoremstyle{definition}

\theoremstyle{remark}
\newtheorem{example}{Example}
\newtheorem{remark}[theorem]{Remark}
\newtheorem*{acknowledgement*}{Acknowledgement}

\begin{document}

\maketitle

\begin{abstract}
We describe the   $(\alpha,\beta)$-metrics whose the $T$-tensor vanishes ($T$-condition) and the  $(\alpha,\beta)$-metrics that satisfy the $\sigma T$-condition $\sigma_hT^h_{ijk}=0$, where  $\sigma_h=\frac{\partial \sigma}{\partial x^h}$  and $\sigma$ is a smooth function on $M$. These  classes have already been   obtained by Z. Shen and G. S. Asanov in a completely different approach. The Finsler metrics of the  first class are  Berwaldian, the metrics of the second class are  almost regular   non-Berwaldian Landsberg metrics.
\end{abstract}

\section{Introduction}

 The   $T$-tensor plays an interesting role in Finsler geometry and general relativity. It was introduced by M. Matsumoto \cite{matsumoto}. M. Hashiguchi \cite{hashiguchi} showed that a Landsberg space remains a Landsberg space under all conformal changes of the Finsler function if and only if its $T$-tensor vanishes. By a famous observation of Z. I. Szab\'o \cite{Szabo_$T$-tensor}, a positive definite Finsler manifold with vanishing $T$-tensor is Riemannian. For further information, we refer to the papers \cite{r2.12,r2.8,matsumoto}. Moreover, for the physical  point of view, we refer, for example, to \cite{Asanov2,Asanov,Asanov_$T$-tensor}

Let  $(M,F)$ be a Finsler manifold. We recall that a conformal change $F\rightsquigarrow \overline{F}$ of $F$ by a smooth function $\sigma$ on $M$ is given by
\begin{equation}\label{conformal_chnage}
\overline{F}(v):=e^{\sigma(p)}F(v) \quad \text{if}\,\,\, v\in T_pM.
\end{equation}
 A Landsberg manifold remains  of the same type under  a conformal change \eqref{conformal_chnage}  if and only if the $T$-tensor satisfies the condition
 $$\sigma_r T^r_{jk\ell}=0, \quad \sigma_r:=\frac{\partial \sigma}{\partial x^r}.$$
  Obviously, if this holds for \textit{every} $\sigma\in C^\infty(M)$,     then $T=0$ and $(M,F)$ is Riemannian by Szab\'o's observation.
  So it will be more beneficial to consider the case when a Landsberg space remains  Landsberg under \textit{some} conformal transformation.
  In \cite{Elgendi_BL},  it was studied in the case when the condition $\sigma_r T^r_{jkh}=0$ is satisfied for some  conformal change by   $\sigma$ on $ M$.

  \bigskip

In this paper, we study the $T$-tensor of the $(\alpha,\beta)$-metrics. An $(\alpha,\beta)$-metric $F$ is of  the form $F=\alpha\phi(s)$, $s:=\frac{\beta}{\alpha}$.   We start by studying the Cartan tensor $C_{ijk}$ of $(\alpha,\beta)$-metrics. We show that the Cartan tensor $C_{ijk}$ vanishes identically and hence the space is Riemannian if and only if $\phi(s)= \sqrt{k_1+k_2s^2}$, where $k_1$ and $k_2$ are constants.

We calculate the $T$-tensor for the $(\alpha,\beta)$-metrics, and we find  necessary and sufficient conditions for $(\alpha,\beta)$-metrics to satisfy the $T$-condition. By solving some ODEs, we  show that an  $(\alpha,\beta)$-metric satisfies the $T$-condition if and only if it is Riemannian or $\phi(s)$ has the following form
 $$\phi(s)=
      c_3s^{\frac{c b^2-1}{cb^2}}(cb^2-cs^2)^{\frac{1}{2cb^2}}. $$

We introduce the notion of $\sigma T$-condition. We say that a Finsler space satisfies this condition if it  admits smooth function  $\sigma(x)$ such that $\sigma_hT^h_{ijk}=0$, where $\sigma_h=\frac{\partial \sigma}{\partial x^h}$. We find  necessary and sufficient conditions for an $(\alpha,\beta)$-metric   to satisfy the $\sigma T$-condition. Moreover, we show that  the $(\alpha,\beta)$-metrics satisfy the  $\sigma T$-condition if and only if the $T$-tensor vanishes (this is the trivial case) or $\phi(s)$ is given by
$$\phi(s)=c_3 \,\exp\left(\int_0^s \frac{c_1\sqrt{b^2-t^2}+c_2 t}{t(c_1\sqrt{b^2-t^2}+c_2 t)+1}dt\right). $$

It is worthy  to mention that the above  special $(\alpha,\beta)$-metrics have already been  obtained by Z. Shen \cite{Shen_example}. Namely, the  formulas of $\phi(s)$ that characterized the $T$-condition produce  positively almost regular Berwald metrics. One can predict that the metric is not regular in this case because the $T$-tensor vanishes (by Szab\'o's observation). In his paper, Shen showed this almost regular property.  The non-trivial formula that characterized the $\sigma T$-condition (with some restrictions) provides the class of (almost regular) Landsberg metrics  which are not Berwaldian.

 In \cite{Elgendi_BL}, it was  claimed  that the long existing problem of regular Landsberg  non-Berwaldian spaces  is (closely) related to   the question:\\

 \textit{Is there  any  Finsler space admitting a smooth function $\sigma$ such that $\sigma_rT^r_{ijk}=0$, $\sigma_r=\frac{\partial \sigma}{\partial x^r}$?}\\

   In this paper we confirm this claim in the almost regular case, since the class of $(\alpha,\beta)$-metrics that satisfy the $\sigma T$-condition is the same as the class of  non-Berwaldian Landsberg  metrics  obtained by Z. Shen in his quoted paper \cite{Shen_example}.

\section{The Cartan tensor and $T$-tensor of $(\alpha,\beta)$-metrics}

Let  $M$ be  an n-dimensional smooth manifold.    The tangent space to $M$ at $p$ is denoted by  $T_pM$; $TM:=\undersymm{\bigcup}{p\in M}\,\,T_{p}M$ is the tangent bundle of $M$, $\tau:TM\longrightarrow M$ is the tangent bundle projection. We fix a chart $(\mathcal{U}, (u^1,...,u^n))$ on $M$. It induces a local coordinate system $(x^1,...,x^n, y^1,...,y^n)$ on $TM$, where
$$x^i:=u^i\circ \tau, \quad y^i(v):=v(u^i) \quad (v\in \tau^{-1}(\mathcal{U})).$$
By abuse of notation, we shall denote the coordinate functions $u^i$ also by $x^i$.

Let $\alpha$ be a Riemannian metric, $\beta$ a 1-form on $M$. Locally,
$$\alpha \underset{(\mathcal{U})}{=} a_{ij}\,dx^i\otimes dx^j, \quad \beta \underset{(\mathcal{U})}{=} b_{i}\,dx^i.$$

The Riemannian metric $\alpha$ induces naturally a Finsler function $F_\alpha$ on $TM$ given by $F_\alpha(v):=\sqrt{\alpha_{\tau(v)}(v,v)}$. Similarly, the 1-form $\beta$ can be interpreted as a smooth function
$$\overline{\beta}:TM\longrightarrow \Real, \quad v\longmapsto \overline{\beta}(v):=\beta_{\tau(v)}(v).$$
Locally,
$$F_\alpha\underset{(\mathcal{U})}{=}\sqrt{(a_{ij}\circ\tau ) y^iy^j}, \quad \overline{\beta}\underset{(\mathcal{U})}{=}(b_i\circ \tau )y^i.$$
In what follows, as usual, we shall simply write $\alpha$ and $\beta$ instead of $F_\alpha$ and $\overline{\beta}$, respectively.

For any $p\in M$, we define
$$\| \beta_p\|_\alpha:=\sup_{v\in T_pM\backslash\{0_p\}}\frac{\beta(v)}{\alpha(v)}.$$

An  $(\alpha,\beta)$-metric for $M$ is a function $F$ on $\T M:=\undersymm{\bigcup}{p\in M}(T_pM\backslash\{0_p\})$ defined by
$$F:=\alpha \phi(s):=\alpha(\phi\circ s), \quad s:=\frac{\beta}{\alpha},$$
where $\phi :(-b_0,b_0)\longrightarrow \Real$ is a smooth function $(b_0>0)$.

Now suppose that $\| \beta_p\|_\alpha< b_0$ for any $p\in M$. Then $F=\alpha(\phi\circ \frac{\beta}{\alpha})$ is a (positive definite) Finsler function if and only if $\phi$ satisfies the following conditions:
\begin{equation}\label{phi_conditions}
\phi(t)>0, \quad \phi(t)-t\phi'(t)+(x^2-t^2)\phi''(t)>0,
\end{equation}
where $t$ and $x$ are arbitrary real numbers with $|t|<x<b_0$. (For a proof, see Shen \cite{Shen_example}, Lemma 2.1) In this case we say that $F$ is a \textit{regular}  $(\alpha,\beta)$-metric. If $\| \beta_p\|_\alpha\leq b_0$ for all $p\in M$, then $F=\alpha(\phi\circ\frac{\beta}{\alpha})$ is called \textit{almost regular} (under condition \eqref{phi_conditions}). An almost regular  $(\alpha,\beta)$-metric $F=\alpha(\phi\circ\frac{\beta}{\alpha})$ is \textit{positively almost regular} if $\phi$ is defined only on $(0,b_0)$.

For an $(\alpha,\beta)$-metric $F=\alpha\phi(s)$, the components $g_{ij}=\frac{1}{2}\frac{\partial^2}{\partial y^i\partial y^j} F^2$ of the fundamental tensor can be calculated by the formula
\begin{equation}\label{a_ij}
g_{ij}=\rho a_{ij}+\rho_0 b_ib_j+\rho_1(b_i\alpha_j+b_j\alpha_i)+\rho_2 \alpha_i\alpha_j,
\end{equation}
where $\alpha_i:=\frac{\partial \alpha}{\partial y^i}=\frac{(a_{ij}\circ \tau )}{\alpha}y^j$ and
\begin{eqnarray*}\label{Chern-Shen}
      \rho&:=&\phi^2-s\phi\phi', \\
     \rho_0&:=&\phi'^2+\phi \phi'',\\
               \rho_1&:=&\phi\phi'-s(\phi'^2+\phi \phi''),\\
    \rho_2&:=&s^2(\phi'^2+\phi \phi'')-s\phi\phi',
\end{eqnarray*}
see Chern-Shen  \cite{Shen-book}, p. 179, where $b^i=a^{ij}b_j$.

Moreover, we have
\begin{equation}\label{det(g)}
\det(g_{ij})=\phi^{n+1}(\phi-s\phi')^{n-2}\left((\phi-s\phi')+(b^2-s^2)\phi''\right)\det(a_{ij}),
\end{equation}
where $ b^2:=b^ib_i$.

The formula for the inverse metric $g^{ij}$ can be found in \cite{Shen-book} as follows.
\begin{proposition} \label{inverse-metric}
For an $(\alpha,\beta)$-metric  $F=\alpha\phi(s)$,  the inverse $(g^{ij})$ of the matrix $(g_{ij})$ is given by
$$g^{ij}=\frac{1}{\rho}a^{ij}+\mu_ob^ib^j+\mu_1(b^i\alpha^j+b^j\alpha^j)+\mu_2\alpha^i\alpha^j,$$
where
$\mu_o:=-\frac{\phi\phi''}{\rho(\rho+\phi\phi''m^2)}$,\quad$\mu_1:=-\frac{\rho_1}{\rho(\rho+\phi\phi''m^2)}$, \,$\mu_2:=\frac{\rho_1(s\rho+(\rho_1+s\phi\phi'')m^2)}{\rho^2(\rho+\phi\phi''m^2)}$ and $m^2:=b^2-s^2$.
\end{proposition}

 \begin{remark}\label{denomaintor}
It should be noted that the choice $\phi(s)=c_1s+c_2\sqrt{b^2-s^2}, \, c_1 \,\, \text{and} \, c_2  \,\text{are  constants} $  is excluded. Indeed, the function $\rho+\phi\phi''m^2$ appearing in the denominators of $\mu_0$, $\mu_1$ and $\mu_2$ can be written as follows
$$\rho+\phi(s)\phi''(s)m^2=\phi(s)(\phi(s)-s\phi'(s)+b^2-s^2\phi''(s)).$$
 So $\rho+\phi\phi''m^2=0$ yields
 $$\phi(s)(\phi(s)-s\phi'(s)+(b^2-s^2)\phi''(s))=0,$$
 which contradicts to  condition \eqref{phi_conditions}. To avoid not only this contradiction, but also the dividing by zero (in  $\mu_0$, $\mu_1$ and $\mu_2$),
 we must exclude  the choice of $\phi$  for which $\rho+\phi\phi''m^2=0$.
 Since $\phi$ cannot be zero, we have
  $$\phi(s)-s\phi'(s)+(b^2-s^2)\phi''(s)=0.$$
The solution of this ODE is the function
$$\phi(s)=c_1s+c_2\sqrt{b^2-s^2},$$
where  $ c_1$  {and} $ c_2$  are  constants.
  \end{remark}

It should be noted that, in the literature,  the metric $F=\alpha \phi(s)$, $\phi(s)= k_1s+k_2\sqrt{1+k_3s^2}$, $k_1>0$ is a Finsler metric of Randers-type. But with certain choice of the constant $k_3$, we can get the case where the metric tensor is singular ( $\det(g_{ij})=0$). For example,
\begin{example}
Let $M=\Real^n$,  $\alpha=|y|$ and $\beta=\varepsilon y^1$, $\varepsilon$ is a constant. Then, we have
$$a_{ij}=\delta_{ij},\quad b^2=\varepsilon^2.$$
Then the metric  $F=\alpha \phi(s)$, $\phi(s)=c_1 s+c_2 \sqrt{\varepsilon^2 -s^2}$, by \eqref{det(g)}, is singular in the sense that its metric tensor has vanishing determinant.
\end{example}
\medskip

\begin{lemma}\label{C_ijk}
The components $C_{ijk}=\frac{1}{2}\frac{\partial g_{ij}}{\partial y^k}$ of the Cartan tensor  of an $(\alpha,\beta)$-metric are given by
$$C_{ijk}=\frac{\rho_1}{2\alpha}(h_{ij}m_k+h_{jk}m_i+h_{ik}m_j)+\frac{\rho_0'}{2\alpha}m_im_jm_k,$$
where $h_{ij}=a_{ij}-\alpha_i\alpha_j$ and $m_i:=b_i-s\alpha_i$.
\end{lemma}

\begin{proof}
Differentiating \eqref{a_ij} with respect to $y^k$ and  taking into account  that $\frac{\partial s}{\partial y^k}=\frac{m_k}{\alpha}$, we have
\begin{eqnarray*}
2C_{ijk}&=&\frac{\rho'}{\alpha}a_{ij}m_k+\frac{\rho_0'}{\alpha}b_ib_jm_k+\frac{\rho_1'}{\alpha}(b_i\alpha_j+b_j\alpha_i)m_k+\frac{\rho_1}{\alpha}(b_ih_{jk}+b_jh_{ik})\\
&&+\frac{\rho_2'}{\alpha}\alpha_i\alpha_jm_k
+\frac{\rho_2}{\alpha}(\alpha_ih_{jk}+\alpha_jh_{ik}).
\end{eqnarray*}
Since
$$\frac{\partial \alpha_j}{\partial y^k}=\frac{1}{\alpha}h_{jk}, \,\,\,  \rho'=\rho_1, \,\,\, \rho_1'=-s\rho_0', \,\,\, \rho_2'=s^2\rho_0'-\rho_1$$
 the result follows.
\end{proof}

\begin{remark}\label{m_i} The covariant vector $m_i$ satisfies the  properties
$$m_i\neq 0, \quad y^im_i=0, \quad  m^2=m^im_i=b^im_i\neq 0, \quad b^ih_{ij}=m_j, $$
where  $m^2=b^2-s^2$.
\end{remark}

\begin{lemma}\label{zeta-eta}
Let $(M,F)$ be   an $(\alpha,\beta)$-metric   with $n\geq 3$  such that
$$\zeta(h_{ij}m_k+h_{jk}m_i+h_{ik}m_j)+\eta m_im_jm_k=0,$$
where $\zeta(x,y)$ and $ \eta(x,y) $ are smooth functions on $\T M$. Then $\zeta$ and $ \eta$ must vanish.
\end{lemma}

\begin{proof}
Assume that
$$\zeta(h_{ij}m_k+h_{jk}m_i+h_{ik}m_j)+\eta m_im_jm_k=0.$$
Contracting the above equation by $b^ib^j$ and using Remark \ref{m_i}, we obtain
\begin{equation}\label{eq-1}
3\zeta+\eta  m^2=0.
\end{equation}
And the contraction by $g^{ij}$ gives
\begin{equation}\label{eq-2}
(n+1)\zeta+ \eta m^2=0.
\end{equation}
Now,  taking the fact that $n\geq 3$,  subtracting \eqref{eq-1}  and \eqref{eq-2} we get  $\zeta=0$ and $\eta=0$.
\end{proof}

\begin{lemma}\label{Theorem_3.3}
Let $(M,F)$ be   an $(\alpha,\beta)$-metric  with $n\geq 3$.   If there exist covectors  $A_i$ and $B_j$ on $TM$ such that $y^iA_i=0$, $y^iB_i=0$ and  the following combination is satisfied
$$ h_{ij}A_k+h_{jk}A_i+h_{ik}A_j+B_im_jm_k+B_jm_im_k+B_km_im_j=0,$$
  then $A_i$ and $ B_i$ must vanish at each point of $TM$, that is,  $A_i$ and $ B_i$ are zero covectors.
\end{lemma}

\begin{proof}
Assume that
$$ h_{ij}A_k+h_{jk}A_i+h_{ik}A_j+B_im_jm_k+B_jm_im_k+B_km_im_j=0.$$
Contracting the above equation by $b^ib^j$ and using Remark \ref{m_i}, we obtain
\begin{equation}\label{Eq:1}
2(A_\beta+m^2B_\beta)m_k+m^2(A_k+m^2B_k)=0,
\end{equation}
where we use the notations $A_\beta:=A_ib^i$ and $B_\beta :=B_ib^i$.  Using the facts that  $y^iA_i=0$, $y^iB_i=0$, the contraction by $a^{ij}$ gives
\begin{equation}\label{Eq:2}
(n+1)A_k+ 2B_\beta m_k+m^2 B_k=0.
\end{equation}
Again, contracting the equations \eqref{Eq:1} and \ref{Eq:2} by $b^k$ gives rise to
\begin{equation}\label{Eq:3}
 A_\beta+m^2B_\beta =0,
\end{equation}
\begin{equation}\label{Eq:4}
(n+1)A_\beta+ 3m^2 B_\beta =0.
\end{equation}
Multiplying $\eqref{Eq:3}$ by $3$ and subtracting it from \eqref{Eq:4}, then using the fact that $n>2$, we get that $A_\beta=0$, $B_\beta=0$. By substitution into \eqref{Eq:1} and \eqref{Eq:2} and repeating the last process we obtain that $A_k=0$ and $B_k=0$.
\end{proof}

By the help of Lemma \ref{zeta-eta}, one can easily prove the following theorem.
\begin{theorem}For the $(\alpha,\beta)$-metrics with  $n\geq 3$, the following assertions are equivalent:
\begin{itemize}
\item[(a)] $\rho_1=0$.

\item[(b)] $\rho_2=0$.

\item[(c)] $(\alpha,\beta)$-metric is Riemannian.

\item[(d)] $\phi=\sqrt{k_1s^2+k_2}$.
\end{itemize}
\end{theorem}

For a Finsler manifold $(M,F)$, the $T$-tensor is defined by \cite{ttensor}
\begin{equation} \label{$T$-tensor}
T_{rijk}=FC_{rijk}-F(C_{sij}C^{s}_{rk}+C_{sjr}C^{s}_{ik}+C_{sir}C^{s}_{jk})
+C_{rij}\ell_k+C_{rik}\ell_j +C_{rjk}\ell_i+C_{ijk}\ell_r,
\end{equation}
where $\ell_j:=\dot{\partial}_j F$, $ C_{rijk}:=\dot{\partial}_rC_{ijk}$ and  $\dot{\partial}_j$ is the differentiation with respect to $y^j$.
The $T$-tensor is totally symmetric in all of its indices.

\begin{theorem}\label{t}
The $T$-tensor of an $(\alpha,\beta)$-metric  takes the form:
\begin{eqnarray*}
  {T}_{hijk} &=& \Phi (h_{hi}h_{jk}+h_{hj}h_{ik}+h_{hk}h_{ij})\\
  &&+\Psi(h_{hk}m_im_j+h_{hj}m_im_k+h_{hi}m_jm_k+h_{ij}m_hm_k+h_{jk}m_im_h+h_{ik}m_jm_h)\\
   &&+\Omega m_hm_im_jm_k
\end{eqnarray*}
where
$$\Phi:=-\frac{\rho_1\phi}{2\alpha}
  (s+\alpha K_1m^2),\quad \Psi:=\frac{\rho_1\phi'}{\alpha}-\frac{\rho_1^2\phi}{\alpha\rho}-\frac{s\rho_0'\phi}{2\alpha}-\frac{\rho_1\phi m^2K_2}{2},$$
 $$\Omega:=\frac{\rho_0''\phi}{2\alpha}+\frac{2\rho_0'\phi'}{\alpha}-3\phi(k_2(\rho_1+\frac{\rho_0'm^2}{2})+\frac{\rho_1\rho_0'}{2\alpha\rho}),$$
  $$K_1:=\frac{\rho_1(1+\rho\mu_0m^2)}{2\alpha\rho}=\frac{\rho_1}{2\alpha(\rho+m^2\phi\phi'')}, $$
  $$  K_2:=\frac{\rho_0'(1+\rho\mu_0m^2)}{2\alpha\rho}+\frac{\rho_1\mu_o}{\alpha}=\frac{\rho\rho_0'-2\rho_1\phi\phi''}{2\alpha\rho(\rho+m^2\phi\phi'')}.$$
\end{theorem}
\begin{proof}
By using Lemma \ref{C_ijk} and making use of the fact that $\dot{\partial}_is=\frac{m_i}{\alpha}$, we have
\begin{eqnarray}\label{lamda}
  \dot{\partial}_h{C}_{ijk}  &=&
\nonumber     -\frac{\rho_1}{2\alpha^2}(h_{ik}n_{jh}+h_{jk}n_{ih}+h_{ij}n_{kh}
   +h_{jh}n_{ik}+h_{kh}n_{ij}+h_{ih}n_{jk}) \\
\nonumber    &&-\frac{s\rho_1}{2\alpha^2}(h_{ik}h_{jh}+h_{jk}h_{ih}+h_{kh}h_{ij})
   -\frac{s\rho_0'}{2\alpha^2}(h_{ij}m_hm_k+h_{ki}m_jm_h+h_{hk}m_im_j\\
\nonumber    &&+h_{jh}m_km_i+h_{ih}m_jm_k+h_{kj}m_im_h)+\frac{\rho_0''}{2\alpha^2}m_im_jm_hm_k\\
   &&-\frac{\rho_0'}{2\alpha^2}(n_{ij}m_hm_k+n_{kh}m_im_j)
\end{eqnarray}
where  $n_{ij}:=\alpha_im_j+\alpha_jm_i$. By making use of the fact that  $K_1$ and $K_2$ satisfy
$$\frac{\rho_1}{2\alpha}\left(K_2m^2+\frac{\rho_1}{\alpha\rho}\right)=K_1\left(\frac{\rho_1}{\alpha}+\frac{\rho_0'm^2}{2\alpha}\right),$$
we have
\begin{eqnarray}\label{cc}
\nonumber {C}_{ijr} {C}^r_{hk}+{C}_{jkr} {C}^r_{hi}+{C}_{ikr} {C}^r_{hj} &=& \left(
   \frac{\rho_1K_2m^2}{2\alpha}+\frac{\rho_1^2}{\alpha^2\rho}\right)
   (h_{ij}m_hm_k+h_{ki}m_jm_h+h_{hk}m_im_j\\
   \nonumber&&{\hspace{-4cm}}
+h_{jh}m_km_i+h_{ih}m_jm_k+h_{kj}m_im_h)+3\left( \frac{\rho_1\rho_0'}{2\alpha^2\rho}+K_2\left( \frac{\rho_1}{\alpha}+ \frac{\rho_0'm^2}{2\alpha
   }\right)\right)m_im_jm_hm_k\\
 \nonumber  &&{\hspace{-4cm}}+ \frac{\rho_1K_1m^2}{2\alpha}(h_{ik}h_{jh}+h_{jk}h_{ih}+h_{kh}h_{ij}).
\end{eqnarray}
Since $ \ell_i:=\dot{\partial}_iF= \phi \alpha_i+\phi'm_i$, we get
\begin{eqnarray*}
  C_{hij}\ell_k+C_{hik}\ell_j +C_{hjk}\ell_i+C_{ijk}\ell_h&=& \frac{\rho_0'\phi}{2\alpha}(m_im_jn_{kh}+m_km_hn_{ij})
 +\frac{2\rho_0'\phi'}{\alpha}m_im_jm_hm_k\\
   &&{\hspace{-4cm}}+\frac{\rho_1\phi'}{\alpha}(h_{ij}m_hm_k+h_{ki}m_jm_h+h_{hk}m_im_j+h_{jh}m_km_i+h_{ih}m_jm_k+h_{kj}m_im_h)\\
   &&{\hspace{-4cm}}
   +\frac{\rho_1\phi}{2\alpha} (h_{ik}n_{jh}+h_{jk}n_{ih}+h_{ij}n_{kh}
   +h_{jh}n_{ik}+h_{kh}n_{ij}+h_{ih}n_{jk}).
\end{eqnarray*}
Now, taking the fact that $F=\alpha \phi$ into account, the $T$-tensor of the space $(M,F)$ is given by
\begin{eqnarray*}
  {T}_{hijk} &=&FC_{hijk}-F(C_{sij}C^{s}_{hk}+C_{hjr}C^{s}_{ik}+C_{sih}C^{s}_{jk})
+C_{hij}\ell_k+C_{hik}\ell_j +C_{hjk}\ell_i+C_{ijk}\ell_h\\
  &=& \Phi (h_{hi}h_{jk}+h_{hj}h_{ik}+h_{hk}h_{ij})\\
  &&+\Psi(h_{hk}m_im_j+h_{hj}m_im_k+h_{hi}m_jm_k+h_{ij}m_hm_k+h_{jk}m_im_h+h_{ik}m_jm_h)\\
   &&+\Omega\, m_hm_im_jm_k.
\end{eqnarray*}
\end{proof}
For an $(\alpha,\beta)$-metric, one can calculate $\Phi$, $\Psi$ and $\Omega$ to obtain the formula for its $T$-tensor. Or one can, easily,   use Maple program for these calculations, for example we have the following corollary.
\begin{corollary} The $T$-tensor  of  Kropina metric, $(F=\frac{\alpha}{s}  ,\phi(s)=1/s)$, is given by
\begin{eqnarray*}\label{kropina2}
  {T}_{hijk} &=&
  \frac{2}{\alpha^2b^2s^2}(h_{hi}h_{jk}+h_{hj}h_{ik}+h_{hk}h_{ij})+
  \frac{2}{\alpha b^2 s^3}(h_{hi}m_{j}m_k+h_{hj}m_{i}m_k+h_{ij}m_{h}m_k\\
   &&+h_{jk}m_{i}m_h+h_{hk}m_{i}m_j+h_{ik}m_{j}m_h)+\frac{6 }{\alpha b^2 s^5}m_hm_im_jm_k.
\end{eqnarray*}
The $T$-tensor  of Randers metric, ($F={\alpha}(1+s), \phi(s)=1+s)$, is given by
\begin{eqnarray*}
  {T}_{hijk} &=& -
  \frac{b^2+s^2+2s}{4\alpha}(h_{hi}h_{jk}+h_{hj}h_{ik}+h_{hk}h_{ij}),
\end{eqnarray*}
\end{corollary}
It is to be noted  that the $T$-tensor  of  Kropina metric is  also obtained by
Shibata \cite{r2.12} and \cite{beta2}. The $T$-tensor  of Randers metric has been studied by Matsumoto \cite{r2.8}.

\section{The $T$-condition and $\sigma$$T$-conditions}

The Finsler spaces with vanishing $T$-tensor are called Finsler spaces satisfying the \textit{$T$-condition}, for example, see \cite{Asanov_$T$-tensor}. In a similar manner, we will call the Finsler spaces admitting a function  $\sigma (x)$ such that $\sigma_hT^h_{ijk}=0$, $\sigma_h:=\frac{\partial \sigma}{\partial x^h} $   Finsler spaces satisfying   the \textit{$\sigma$$T$-condition}.
In this section, we characterize the $(\alpha,\beta)$-metrics which satisfy the $T$-condition and the $\sigma$$T$-condition.
\begin{theorem}\label{$T$-condition}
The $(\alpha,\beta)$-metrics with $n\geq 3$ satisfy the $T$-condition if and only if $\Phi=0$.
\end{theorem}

\begin{proof}
 Let $T_{hijk}=0$, then we have
 \begin{eqnarray}\label{T=0}
\nonumber && \Phi (h_{hi}h_{jk}+h_{hj}h_{ik}+h_{hk}h_{ij})+\Psi(h_{hk}m_im_j+h_{hj}m_im_k+h_{hi}m_jm_k\\
  &&+h_{ij}m_hm_k+h_{jk}m_im_h+h_{ik}m_jm_h)
   +\Omega\, m_hm_im_jm_k=0.
\end{eqnarray}
Contracting the above equation by $b^h$, we get
  $$(\Phi+m^2\Psi) (h_{jk}m_i+h_{ik}m_j+h_{ij}m_k)+(3\Psi+m^2\Omega)m_im_jm_k=0.$$
  Since $n\geq 3$, Lemma \ref{zeta-eta} implies
\begin{eqnarray}\label{T=0_1}
  \Phi+m^2\Psi=0,\quad 3\Psi+m^2\Omega=0.
\end{eqnarray}
Again, contraction \eqref{T=0} by $a^{hi}$, we obtain
$$
  ((n+1)\Phi+m^2\Psi) h_{jk}+((n+3)\Psi+m^2\Omega)m_jm_k=0.
$$
Then, taking the fact that $n\geq 3$ into account, we get
\begin{eqnarray}\label{T=0_2}
  (n+1)\Phi+m^2\Psi=0, \quad (n+3)\Psi+m^2\Omega=0.
\end{eqnarray}
 Now, solving the equations \eqref{T=0_1} and \eqref{T=0_2} for $\Phi$, $\Psi$ and $\Omega$, we have $\Phi=0$, $\Psi=0$ and $\Omega=0$.

Conversely, let $\Phi=0$, then we have either $\rho_1=0$ or $s+\alpha k_1m^2=0$. If $\rho_1=0$ (the space is Riemannian), then $\rho_0'=0$ and hence $\Psi=0$ and $\Omega=0$. And if  $s+\alpha K_1m^2=0$, one can conclude  that $\Psi=0$ and $\Omega=0$ (see the proof of  Theorem \ref{Theorem_$T$-condition}).
\end{proof}

\begin{proposition}\label{raised-tensor} The $T$-tensor $T^h_{ijk}:=g^{hr}T_{rijk}$ is given by
\begin{eqnarray*}
T^h_{ijk}&=&\frac{\Phi}{\rho}(h^h_{i}h_{jk}+h^h_{j}h_{ik}+h^h_{k}h_{ij})+\frac{\Psi}{\rho}(h^h_{k}m_im_j+h^h_{j}m_im_k+h^h_{i}m_jm_k+h_{ij}m^hm_k\\
&&+h_{jk}m_im^h+h_{ik}m_jm^h)+\frac{\Omega}{\rho} m^hm_im_jm_k+(\mu_0b^h+\mu_1\alpha^h)(\Phi(h_{ik}m_{j}+h_{ij}m_{k}+h_{jk}m_{i})\\
&&+\Psi(m^2(h_{ik}m_{j}+h_{ij}m_{k}+h_{jk}m_{i})+3m_im_jm_k)+\Omega m^2 m_im_jm_k)
\end{eqnarray*}
\end{proposition}
\begin{proof}
The proof is a straightforward calculations by using Proposition \ref{inverse-metric}.
\end{proof}

\begin{theorem}\label{sigma$T$-condition}
The $(\alpha,\beta)$-metrics with $n\geq 3$ satisfies the $\sigma$$T$-condition if and only if
\begin{itemize}
\item[(a)] $\Phi+m^2\Psi=0.$

\item[(b)] $m^2\Omega+3 \Psi=0.$

\item[(c)]$\sigma_j-\frac{\sigma_0}{ s\alpha}b_j=0$.
\end{itemize}
\end{theorem}

\begin{proof}  By using Proposition \ref{raised-tensor}, we have
\begin{eqnarray*}
\sigma_h T^h_{ijk}&=&\frac{\Phi}{\rho}\left(\left(\sigma_i-\frac{\sigma_0}{\alpha}\alpha_i\right)h_{jk}+\left(\sigma_j-\frac{\sigma_0}{\alpha}\alpha_j\right)h_{ik}+\left(\sigma_k-\frac{\sigma_0}{\alpha}\alpha_k\right)h_{ij}\right)\\
&&+\frac{\Psi}{\rho}\Big(\left(\sigma_k-\frac{\sigma_0}{\alpha}\alpha_k\right)m_im_j+\left(\sigma_j-\frac{\sigma_0}{\alpha}\alpha_j\right)m_im_k+\left(\sigma_i-\frac{\sigma_0}{\alpha}\alpha_i\right)m_jm_k\\
&&+\left(\sigma_\beta-s\frac{\sigma_0}{\alpha}\right)\left(h_{ij}m_k
+h_{jk}m_i+h_{ik}m_j\right)\Big)
+\frac{\Omega}{\rho} \left(\sigma_\beta-s\frac{\sigma_0}{\alpha}\right)m_im_jm_k\\
&&+\left(\mu_0\sigma_\beta+\mu_1\frac{\sigma_0}{\alpha}\right)\Big(\Phi(h_{ik}m_{j}+h_{ij}m_{k}+h_{jk}m_{i})\\
&&+\Psi(m^2(h_{ik}m_{j}+h_{ij}m_{k}+h_{jk}m_{i})+3m_im_jm_k)+\Omega m^2 m_im_jm_k\Big),
\end{eqnarray*}
where $\sigma_0:=\sigma_iy^i$ and $\sigma_\beta:=\sigma_ib^i$.
Using the fact that $m_i=b_i-s\alpha_i$, we get
\begin{eqnarray*}
\sigma_h T^h_{ijk}&=&\frac{\Phi}{\rho}\left(\left(\sigma_i-\frac{\sigma_0}{s\alpha}b_i\right)h_{jk}+\left(\sigma_j-\frac{\sigma_0}{s\alpha}b_j\right)h_{ik}+\left(\sigma_k-\frac{\sigma_0}{s\alpha}b_k\right)h_{ij}\right)\\
&&
\frac{\sigma_0\Phi}{s\alpha\rho}(m_ih_{jk}+m_jh_{ik}+m_kh_{ij})\\
&&+
\frac{\Psi}{\rho}\left(\left(\sigma_k-\frac{\sigma_0}{s\alpha}b_k\right)m_im_j+\left(\sigma_i-\frac{\sigma_0}{s\alpha}b_i\right)m_km_j+\left(\sigma_j-\frac{\sigma_0}{s\alpha}b_j\right)m_im_k\right)\\
&&+\frac{\Psi}{\rho}\left(\sigma_\beta-s\frac{\sigma_0}{\alpha}\right)(h_{ij}m_k
+h_{jk}m_i+h_{ik}m_j)
\\
&&+3\frac{\sigma_0\Psi}{s\alpha\rho}m_km_im_j+\frac{\Omega}{\rho} \left(\sigma_\beta-s\frac{\sigma_0}{\alpha}\right)m_im_jm_k\\
&&+\left(\mu_0\sigma_\beta+\mu_1\frac{\sigma_0}{\alpha}\right)(\Phi(h_{ik}m_{j}+h_{ij}m_{k}+h_{jk}m_{i})\\
&&+\Psi(m^2(h_{ik}m_{j}+h_{ij}m_{k}+h_{jk}m_{i})+3m_im_jm_k)+\Omega m^2 m_im_jm_k).
\end{eqnarray*}
The above equation can be written in the following  form
\begin{eqnarray*}
\sigma_h T^h_{ijk}&=&\left(\frac{\sigma_0\Phi}{s\alpha\rho}+\frac{\Psi}{\rho}\left(\sigma_\beta-s\frac{\sigma_0}{\alpha}\right)+(\Phi+m^2\Psi)\left(\mu_0\sigma_\beta+\mu_1\frac{\sigma_0}{\alpha}\right)\right)(h_{ik}m_{j}+h_{ij}m_{k}+h_{jk}m_{i})\\
&&+ \left(\frac{\Omega}{\rho} \left(\sigma_\beta-s\frac{\sigma_0}{\alpha}\right)+ 3\frac{\sigma_0\Psi}{s\alpha\rho}+(\Omega m^2+3\Psi )\left(\mu_0\sigma_\beta+\mu_1\frac{\sigma_0}{\alpha}\right)\right)m_km_im_j \\
&&+\frac{\Phi}{\rho}\left( h_{jk}\left(\sigma_i-\frac{\sigma_0}{s\alpha}b_i\right)+h_{ij}\left(\sigma_k-\frac{\sigma_0}{s\alpha}b_k\right)+ h_{ik}\left(\sigma_j-\frac{\sigma_0}{ s\alpha}b_j\right)\right)\\
&&+\frac{\Psi}{\rho}\left( m_jm_k\left(\sigma_i-\frac{\sigma_0}{ s\alpha}b_i\right)+m_im_j\left(\sigma_k-\frac{\sigma_0}{s\alpha}b_k\right)+ m_im_k\left(\sigma_j-\frac{\sigma_0}{ s\alpha}b_j\right)\right).\\
\end{eqnarray*}

Putting
$A_i:=A m_i+\Phi \tau_i, \quad B_i:=\frac{B}{3}m_i+\Psi \tau_i,$
where $\tau_i:=\frac{1}{\rho}\left(\sigma_i-\frac{\sigma_0}{ s\alpha}b_i\right)$ and
$$A:=\frac{\sigma_0\Phi}{s\alpha\rho}+\frac{\Psi}{\rho}\left(\sigma_\beta-s\frac{\sigma_0}{\alpha}\right)+(\Phi+m^2\Psi)\left(\mu_0\sigma_\beta+\mu_1\frac{\sigma_0}{\alpha}\right),$$
$$B:=\frac{\Omega}{\rho} \left(\sigma_\beta-s\frac{\sigma_0}{\alpha}\right)+ 3\frac{\sigma_0\Psi}{s\alpha\rho}+(\Omega m^2+3\Psi )\left(\mu_0\sigma_\beta+\mu_1\frac{\sigma_0}{\alpha}\right).$$
By using the above quantities, $\sigma_h T^h_{ijk}$  can be written as follows
$$\sigma_h T^h_{ijk}=h_{ij}A_k+h_{jk}A_i+h_{ik}A_j+B_im_jm_k+B_jm_im_k+B_km_im_j.$$

Now, putting $\sigma_h T^h_{ijk}=0$ and since $y^iA_i=0$, $y^iB_i=0$, one can use Lemma \ref{Theorem_3.3} to conclude that
$$A_i=0, \quad B_i=0.$$
Contracting the above two equations by $b^i$ and then by $\sigma^i:=a^{ij}\sigma_j$, respectively, we have
\begin{equation}\label{4_Eqs}
\begin{split}
Am^2 = -\Phi \tau_\beta ,& \quad   Bm^2=-3\Psi \tau_\beta,\\
Am_\sigma=-\Phi\tau_\sigma, & \quad  Bm_\sigma=-3\Psi\tau_\sigma,
\end{split}
\end{equation}
where $\tau_\beta:=\tau_ib^i$,  $m_\sigma:=m_i\sigma^i$ and $\tau_\sigma:=\tau_i\sigma^i$. Now, we claim that both sides of the  four equalities in \eqref{4_Eqs} must vanish.   We prove this claim via contradiction, so we assume that, for example, the sides of the third equality  are non zero, hence by dividing the first equality on the third one, we can get

$$m^2\tau_\sigma = m_\sigma \tau_\beta.$$
From which, we have
$$\left(b^2-\frac{\beta^2}{ \alpha^2}\right)\left(\sigma^2-\frac{\sigma_0}{ s\alpha}\sigma_\beta\right)=\left(\sigma_\beta-\frac{\beta}{  \alpha^2}\sigma_0\right)\left(\sigma_\beta-\frac{\sigma_0}{ s\alpha}b^2\right),$$
where $\sigma^2:=\sigma_i\sigma^i$. Now, simplifying the above equation we get the following
$$\alpha^2 (b^2\sigma^2-  \sigma_\beta^2-\sigma_0)- \sigma^2 \beta^2+\sigma_0 \sigma_\beta \beta=0.$$
Making use of the facts that $\frac{\partial \sigma_0}{\partial y^i}=\sigma_{i}$ and the functions  $\sigma^2$, $\sigma_\beta$,  $b^2$ are functions on $M$, that is, they are functions of $(x^i)$ only, then differentiating the above equation with respect to $y^i$, we have
$$2\alpha(b^2\sigma^2-  \sigma_\beta^2-\sigma_0)\alpha_i-\alpha^2\sigma_i-2\sigma^2\beta b_i+\beta \sigma_\beta \sigma_i+\sigma_0\sigma_\beta b_i=0.$$

By using the properties $\alpha=\sqrt{a_{ij}y^iy^j}$ and   $\frac{\partial (\alpha\alpha_i)}{\partial y^j}=a_{ij}$, differentiating the above equation with respect to $y^j$ and then by $y^k$, we get
$$\sigma_ia_{jk}+\sigma_ja_{ki}+\sigma_ka_{ij}=0.$$
Contracting the above equation by $a^{ij}$, we get
$$(n+3)\sigma_k=0,$$
which gives    $\sigma_k=0$ and this means that $\sigma$ is constant and this is a contradiction. Consequently, all sides of the equalities in \eqref{4_Eqs} are zero. That is,
$$Am^2 =0, \quad \Phi \tau_\beta=0,    \quad   Bm^2=0, \quad \Psi \tau_\beta=0, $$
$$Am_\sigma=0, \quad \Phi\tau_\sigma=0,    \quad  Bm_\sigma=0, \quad \Psi\tau_\sigma=0.$$
Since $m^2\neq 0$ and $\Phi$, $\Psi$ can not be zero, then $A=0$, $B=0$, $\tau_\beta=0$,  $\tau_\sigma=0$ and hence $\tau_i=0$. In other words, we have
$$\frac{\sigma_0\Phi}{s\alpha\rho}+\frac{\Psi}{\rho}\left(\sigma_\beta-s\frac{\sigma_0}{\alpha}\right)+(\Phi+m^2\Psi)\left(\mu_0\sigma_\beta+\mu_1\frac{\sigma_0}{\alpha}\right)=0,$$
$$\frac{\Omega}{\rho} \left(\sigma_\beta-s\frac{\sigma_0}{\alpha}\right)+ 3\frac{\sigma_0\Psi}{s\alpha\rho}+(\Omega m^2+3\Psi )\left(\mu_0\sigma_\beta+\mu_1\frac{\sigma_0}{\alpha}\right)=0,$$
$$\sigma_k-\frac{\sigma_0}{s\alpha}b_k=0.$$
Therefore $\sigma_\beta=\frac{\sigma_0b^2}{s\alpha}$ and  taking the fact that $\sigma\neq 0$ into account,  we get

$$\left(\frac{1}{s\alpha\rho}+\mu_0\frac{b^2}{s\alpha}+\frac{\mu_1}{\alpha}\right)(\Phi+m^2\Psi)=0,$$
$$\left( \frac{1}{s\alpha\rho}+\mu_0\frac{b^2}{s\alpha}+\frac{\mu_1}{\alpha}\right)(\Omega m^2+3\Psi )=0.$$

Now the choice $ \frac{1}{s\alpha\rho}+\mu_0\frac{b^2}{s\alpha}+\frac{\mu_1}{\alpha}=0$ gives the ODE $\rho-s\rho_1-s^2\phi\phi''=0$, which has the solution $\phi=k s$. This solution is just again the same  background Riemannian metric $\alpha$ up to some constants. So, we should have $\Phi+m^2\Psi=0$ and $\Omega m^2+3\Psi =0.$

Conversely, if the conditions (a), (b) and (c) are satisfied, then the result is obviously obtained.
\end{proof}

\begin{remark}\label{Remark_sigm}
The condition
$$\sigma_j-\frac{\sigma_0}{ s\alpha}b_j=0$$
is equivalent to $\sigma_j=e^{a(x)}b_j$. Indeed,
$$\sigma_j=\frac{\sigma_0}{ s\alpha}b_j\Longleftrightarrow \dot{\partial}_j \ln \sigma_0=\dot{\partial}_j \ln \beta\Longleftrightarrow \sigma_0=e^{a(x)}\beta\Longleftrightarrow\sigma_j=e^{a(x)}b_j,$$
where $a(x)$ is an arbitrary, locally defined function on $M$.
\end{remark}

\section{Some ODEs}

In this section, we  focus our study  on the $T$-condition and $\sigma$$T$-condition. By solving some ODEs, we find explicit formulas for $(\alpha,\beta)$-metrics that satisfy the $T$-condition and $\sigma$$T$-condition.

We define a function $Q(s)$ as follows
$$Q(s):=\frac{\phi'}{\phi-s\phi'} .$$
The function $Q(s)$ simplifies and helps to solve the ODEs that will be treated in this section. Moreover, $\phi$ is given by
\begin{equation}\label{Q}
\phi(s)=\exp\left(\int_0^s \frac{Q}{1+tQ}dt\right).
\end{equation}
\begin{theorem}\label{Theorem_$T$-condition}
An $(\alpha,\beta)$-metric with $n\geq 3$ satisfies the $T$-condition if and only if it is Riemannian or  $\phi$ is given by
\begin{equation}\label{berwald}
\phi(s)=
c_3s^{\frac{c b^2-1}{cb^2}}(cb^2-cs^2)^{\frac{1}{2cb^2}}.
\end{equation}
\end{theorem}

\begin{proof}
By Theorem \ref{$T$-condition}, any  $(\alpha,\beta)$-metric satisfies the $T$-condition if and only if $\Phi=0$.
So,  taking the fact that $\phi-s\phi\neq 0$ into account, the ODE $s+\alpha K_1m^2=0$ can be rewritten as follows
$$Q'+\left(\frac{1}{s}+\frac{2s}{m^2}\right)Q=-\frac{2}{m^2}.$$
This is a first order linear differential equation and has the solution
$$Q=\frac{c\,b^2-1}{s}-cs=\frac{c\,(b^2-s^2)-1}{s}, \quad c \,\,\text{is a constant}.$$
Hence,
$$1+sQ=c\,b^2-cs^2, \quad \frac{Q}{1+sQ}=\frac{1}{s}-\frac{1}{c s(b^2-s^2)}.$$
By using \eqref{Q}, $\phi(s)$ is given by \eqref{berwald}. Plugging  $\phi(s)$ in  $\Psi$ and $\Omega$, we have $\Psi=0$ and $\Omega=0$.
\end{proof}

\begin{theorem}
An $(\alpha,\beta)$-metric with $n\geq 3$ satisfies the $\sigma$$T$-condition if and only if it satisfies the $T$-condition or $\phi$ is given by
\begin{equation}\label{Landsberg}
\phi(s)=c_3 \,\exp\left(\int_0^s \frac{c_1\sqrt{b^2-t^2}+c_2 t}{t(c_1\sqrt{b^2-t^2}+c_2 t)+1}dt\right).
\end{equation}
\end{theorem}
\begin{proof}
First we should write $\Phi$ and $\Psi$ in terms of $Q(s)$ and its derivations with respect to $s$, as  follows
$$\Phi=-\frac{\phi (\phi-s\phi')^2(Q-sQ')(sm^2 \phi' Q'+(2s\phi+m^2\phi')Q)}{4 \alpha(m^2 \phi' Q'+\phi Q)},$$
$$\Psi=-\frac{\phi (\phi-s\phi')^2Q''(sm^2 \phi' Q'+(2s\phi+m^2\phi')Q)}{4 \alpha(m^2 \phi' Q'+\phi Q)}.$$
Now, making use of the condition \eqref{phi_conditions}, Remark \ref{denomaintor} and the fact that $\phi-s\phi'\neq 0$, the condition $\Phi+m^2\Psi=0$ gives the following two possible ODEs
\begin{equation}\label{nontrivial_case}
(b^2-s^2)Q''-sQ'+Q=0
\end{equation}
or
\begin{equation}\label{trivial_case}
sm^2\phi'Q'+2s\phi Q+m^2\phi'Q=0
\end{equation}
The ODE \eqref{trivial_case} can be given in the form
$$Q'+\left(\frac{1}{s}+\frac{2s}{m^2}\right)Q=-\frac{2}{m^2}$$
which gives the trivial case, that is, the $T$-tensor vanishes. The ODE \eqref{nontrivial_case} has the solution
$$Q(s)=c_1s+c_2\sqrt{b^2-s^2}.$$
By using \eqref{Q}, $\phi(s)$ is given by  \eqref{Landsberg}.
\end{proof}

\section{Examples and Concluding remarks}

We start by giving two classes  of examples  satisfying the $\sigma$$T$-condition.

\begin{example} Let $M=\Real^n$ and    $\alpha$, $\beta$ be given   by
$$\alpha=f(x^1)|y|, \quad \beta=f(x^1)y^1, $$  where $|y|$ is the Euclidean norm and $f(x^1)$ is arbitrary function on $M$. Then, the class
$$F=\sqrt{\alpha^2+p\beta\sqrt{\alpha^2-\beta^2}+q\beta^2}\,e^{\frac{p}{\sqrt{p^2-4q-4}}\, \text{arctanh}\left(\frac{p\beta+2\sqrt{\alpha^2-\beta^2}}{\beta\sqrt{p^2-4q-4}}\right)}$$
satisfies the $\sigma$$T$-condition, where $p$ and $q$ are arbitrary  constants. Indeed, in this class, one can see that the function $\phi(s)$ is given by
$$\phi(s)=\sqrt{1+ps\sqrt{1-s^2}+qs^2}\,e^{\frac{p}{\sqrt{p^2-4q-4}}\, \text{arctanh}\left(\frac{ps+2\sqrt{1-s^2}}{s\sqrt{p^2-4q-4}}\right)}.$$
Using the formula of $\phi(s)$,  it is much simpler to use the Maple program to show that
$$\Phi+m^2\Psi=0, \quad  m^2\Omega+3 \Psi=0.$$
Moreover, since $\beta= f(x^1)y^1$, then we have $b_1=f(x^1)$, $b_2=\cdots=b_n=0$ and taking Remark \ref{Remark_sigm} into account,    $\sigma_1=\frac{\partial\sigma}{\partial x^1}=\omega(x^1)f(x^1)$, for some   function $\omega(x^1)$    on $M$, therefore one can see that  $\sigma(x)=\theta(x^1)$ where $\theta(x^1)$ is an arbitrary function on $M$. Another way, one can use the Finsler package and Maple program to calculate the $T$-tensor, but in this case we have to choose the dimension, say $n=3$, then one can find that
$$T^1_{ijk}=0, \quad \text{for all} \,\, i,j,k =1,2, 3.$$
And since, $\sigma=\theta(x^1)$, then $\sigma_1=\frac{\partial\theta}{\partial x^1}$ and hence
$$\sigma_1T^1_{ijk}=\frac{\partial\theta}{\partial x^1}T^1_{ijk}=0.$$
\end{example}

\begin{example}
Let $M= \mathbb{R}^3$, and   $\alpha=\sqrt{(y^2)^2+e^{2x^2}((y^1)^2+(y^3)^2)}$, $\beta=y^2$. Then, the class
$$F=\sqrt{\alpha^2+\beta\sqrt{\alpha^2-\beta^2}}\,e^{\frac{1}{\sqrt{3}} \ \arctan\Big(\frac{2\beta}{\sqrt{3}\sqrt{\alpha^2-\beta^2}}+\frac{1}{\sqrt{3}}\Big)}$$
satisfies the $\sigma$$T$-condition. As in the previous example, we repeat the same process. So, one can see that the function $\phi(s)$ is given by
$$\phi(s)=\sqrt{1+s\sqrt{1-s^2}}\,e^{\frac{1}{\sqrt{3}} \ \arctan\Big(\frac{2s}{\sqrt{3}\sqrt{1-s^2}}+\frac{1}{\sqrt{3}}\Big)}.$$
Using Maple program, or by hand, we can show that
$$\Phi+m^2\Psi=0, \quad  m^2\Omega+3 \Psi=0.$$
 Since $\beta=  y^2$, then we have $b_2=1$, $b_1=b_3=0$. As in the previous example, we can have   $\sigma(x)=\theta(x^2)$ for some functions $\theta(x^2)$ on $M$. Or instead,  using the Finsler package and Maple program,   we obtain that
$$T^2_{ijk}=0, \quad \text{for all} \,\, i,j,k =1,2, 3.$$
And since, $\sigma=\theta(x^2)$, then $\sigma_2=\frac{\partial\theta}{\partial x^2}$ and hence
$$\sigma_2T^2_{ijk}=\frac{\partial\theta}{\partial x^2}T^2_{ijk}=0.$$
\end{example}

\bigskip
Finally, we have the following remarks:

$\bullet$ Consider the conformal transformation of a Finsler function $F$, that is, $\overline{F}=\kappa(x) F,$
where $\kappa(x)$ is positive smooth function on $M$.
Then,  by simple and straightforward calculations, one can obtain that the $T$-tensor is transformed by the formula
$$\overline{T}^h_{ijk}=\kappa(x)T^h_{ijk}.$$

 In Example 2, one can see that the conformal transformation of $F$ by any positive smooth function $\kappa(x^1)$ still satisfying the $\sigma T$-condition, that is,
$\overline{F}=\kappa(x^1)F$
satisfies the $\sigma T$-condition. Also, in Example 3, the Finsler function
$\overline{F}=\kappa(x^2)F$
satisfies the $\sigma T$-condition.

\medskip

$\bullet$ By the following special choice   $c_2:=-c$ and $c_1:=cb^2-1$ ($b^2$ is constant), the class \eqref{berwald} becomes
$$\phi(s)=c_3s^{\frac{c_1}{c_1+1}}(1+c_1+c_2s^2)^{\frac{1}{2(c_1+1)}}.$$
which is    the  same as the one obtained by \cite{Shen_example} ( (7.4) in Theorem 7.2). Moreover,    this metric is  positively almost regular Berwaldian.

It should be noted that this   irregularity is studied by Z. Shen \cite{Shen_example}. Here we confirm that this metric is not regular Finsler metric because it  has vanishing $T$-tensor. This because of Z. Szab\'o's result, that is, positive definite Finsler metric with vanishing $T$-tensor is Riemannian.

\medskip

$\bullet$ If $b(x)=b_0$, then \eqref{Landsberg} can be rewritten an follows
$$\phi(s)=c_3 \,\exp\left(\int_0^s \frac{c_2'\sqrt{1-(t/b_0)^2}+c_1 t}{t(c_2'\sqrt{1-(t/b_0)^2}+c_1 t)+1}dt\right), \quad c_2':=c_2/b_0. $$
We notice that the above formulae for $\phi$ is the same as  the one obtained in \cite{Shen_example} ( (1.3) in Theorem 1.2). Under some restrictions on  $\beta$, this represents a class of Landsberg non-Berwaldian Finsler spaces. Also, with a special choice of the constants, $b_0=1$ and $c_1=0$, we obtain
$$\phi(s)=c_3 \,\exp\left(\int_0^s \frac{c_2'\sqrt{1-t^2}}{c_2't\sqrt{1-t^2}+1}dt\right), $$
which is obtained by Asanov \cite{Asanov}.

$\bullet$     Summarizing above, the classes   \eqref{berwald} and \eqref{Landsberg}  are almost regular $(\alpha,\beta)$-metrics. Moreover, the class \eqref{Landsberg} of $(\alpha,\beta)$-metrics that satisfies the $\sigma T$-condition, when $b(x)=b_0$ for some constant $b_0$, is the same  as the class  which is obtained by Z. Shen in \cite[Theorem 1.2]{Shen_example}. This confirms our previous claim   in  \cite{Elgendi_BL} that  the long existing problem of regular Landsberg  non-Berwaldian spaces  is (closely) related to   the question:\\

 \textit{Is there  any  Finsler space admitting functions $\sigma_r(x)$ such that $\sigma_rT^r_{ijk}=0$?}

\end{document}